\documentclass[a4paper,11pt]{article}

\usepackage[utf8]{inputenc}
\usepackage{amsmath, amssymb, amsthm}
\usepackage{graphicx}
\usepackage{accents}
\usepackage{color}
\usepackage{geometry}
\usepackage{fancyhdr} 
\usepackage{titling}
\usepackage{enumerate}
\usepackage{titlesec}
\usepackage{authblk}

\titlespacing*{\section}
{0pt}{5.5ex plus 1ex minus .2ex}{4.3ex plus .2ex}
\titlespacing*{\subsection}
{0pt}{5.5ex plus 1ex minus .2ex}{4.3ex plus .2ex}

\newcommand{\subtitle}[1]{%
  \posttitle{%
    \par\end{center}
    \begin{center}\large#1\end{center}
    \vskip0.5em}%
}

\newcommand{\R}{\mathbb R}

\newcommand{\Z}{\mathbb{Z}}

\newcommand{\Leg}{\mathcal{L}}

\newcommand{\const}{\text{const.}}
\newcommand{\Area}{\text{Area}}

\newtheorem{cor}{Corollary}[section]

\newtheorem{prop}[cor]{Proposition}
\newtheorem*{prop*}{Proposition}
\newtheorem{thm}[cor]{Theorem}

\newtheorem{rmk}[cor]{Remark}

\title{A Finsler Geodesic Flow On $T^2$ With Positive Metric Entropy}
\author{Stefan Klempnauer\thanks{stefan.klempnauer@rub.de}}
\affil{Faculty of Mathematics, Ruhr-University Bochum}
\begin{document}
\maketitle

\begin{abstract}
\noindent We use a theorem of P. Berger and D. Turaev to construct an example of a Finsler geodesic flow 
on the 2-torus with a transverse section, such that its Poincar\'e return map has positive metric entropy. The 
Finsler metric generating the flow can be chosen to be arbitrarily $C^\infty$-close to a flat metric. 

\end{abstract}

The author thanks the SFB CRC/TRR 191 \emph{Symplectic Structures in Geometry, Algebra and Dynamics} of the DFG and the Ruhr-University Bochum for the funding of his research.

\section{Introduction}

We give an example of a Finsler geodesic flow on the 2-torus that exhibits aspects of both chaotic and
integrable dynamics. The flow exhibits integrable behaviour in the sense that a a large region of the unit tangent
bundle is foliated by invariant tori, on each of which the geodesic flow is a linear flow. 
It is chaotic in the sense that its return map of a certain 
transverse section contains a stochastic island, i.e. a region of positive metric entropy. \\

In section 2 we use a result of P. Berger and D. Turaev to obtain a perturbation of the standard shear
map of the cylinder with positive metric entropy. In section 3 this map is embedded as a return map of the 
geodesic flow on the unit tangent bundle $ST^2$. 

\section{A twist map $\hat f$ with positive metric entropy}

Let $(M, \omega)$ be a surface with a smooth area form. Let $f: M \to M$ be a diffeomorphism. The 
\emph{maximal Lyapunov exponent} of $x \in M$ is given by 
$$
\lambda(x) = \limsup_{n \to \infty} \frac{1}{n} \log \Vert Df^n(x) \Vert
$$
If $f$ preserves $\omega$ then we have that the \emph{metric entropy} of $f$ is given by
$$
h_\omega(f) := \int_M \lambda(x) d\mu_\omega
$$

From \cite{bergerturaev19} we have the following theorem 

\begin{thm} \emph{(Berger, Turaev '19)}\label{bergerturaev}
Let $(M, \omega)$ be a surface with a smooth area form. If $f: M \to M$ is a smooth area-preserving 
diffeomorphism with a non-hyperbolic periodic point, then there is an arbitrarily $C^\infty$-small
perturbation of $f$, such that the perturbed map $\hat f$ is a smooth, area-preserving diffeomorphism and 
has positive metric entropy $h_\omega(\hat f) > 0$.  

\end{thm}

\begin{rmk} \label{btrmk}
The theorem is proved in \cite{bergerturaev19} by perturbing the diffeomorphism $f$ locally along the orbit of the 
periodic point, such that one obtains an invariant domain $I \subset M$ of positive measure with
positive maximal Lyapunov exponent at every point in $I$. 
Since the perturbation is local the perturbed map $\hat f$ agrees with $f$ away from the orbit of the periodic 
point. The boundary of $I$ consists of finitely many $C^0$-embedded circles that lie in the union of
the stable and unstable manifolds of a set of hyperbolic periodic points. In the spirit of \cite{bergerturaev19} we
call such a set $I$ a 
\emph{stochastic island}.
\end{rmk}

Let $Z = S^1 \times \R$ be the cylinder equipped with the standard symplectic form $dx \wedge dy$. 
The \emph{shear map} $f_1: Z \to Z$ with $(x,y) \mapsto (x+y \mod 1, y)$ is the simplest example of a 
\emph{twist map} of the cylinder. Note that $f_1$ has a non-hyperbolic fixed point at $(0,0)$. 
From theorem \ref{bergerturaev} we obtain an area-preserving (i.e. symplectic because of dimension two) 
$C^\infty$ diffeomorphism $\hat f: Z \to Z$ with a stochastic island $I$, which is arbitrarily close to $f_1$ in 
$C^\infty$. 
The twist of $f_1$ is uniformly equal to 1. Thus, if $\hat f$ is $C^1$-close enough to $f_1$ it has bounded twist
away from zero.
Note that if $\hat f$ is $C^3$-close enough to $f_1$ then KAM-theory guarantees that $\hat f$ possesses 
invariant essential circles and consequently has zero flux (see \cite{herman83}). 
Consequently, $\hat f$ is a twist map of the cylinder if it is close enough to $f_1$ in $C^r$ for $r > 3$.\\

As mentioned in remark \ref{btrmk} there is a $K > 0$, such that 
$\hat f (x,y) = f_1(x,y)$ for every $|y| > K$.

\section{Embedding $\hat f$ into the geodesic flow}

We identify the tangent bundle $TT^2$ with $T^2 \times \R^2$. A \emph{Finsler metric} on $T^2$ is a map 
$$
F : TT^2 \to [0,\infty)
$$
with the following properties
\begin{enumerate}
\item (Regularity) $F$ is $C^\infty$ on $TT^2 - 0$ 
\item (Positive homogeneity) $F(x,\lambda y) = \lambda \cdot F(x,y)$ for $\lambda > 0$
\item (Strong convexity) The hessian 
$$
(g_{ij}(x,v)) = \left(\partial_{v_iv_j} \frac{1}{2} F(x,v)^2 \right)
$$
is positive-definite for every $(x,v) \in TT^2 - 0$. 
\end{enumerate}
A Finsler metric is called \emph{reversible} if $F(x,v) = F(x,-v)$ for every $(x,v) \in TT^2$. The 
\emph{unit tangent bundle} $ST^2 \cong T^2 \times S^1$ is given by $ST^2 = F^{-1}(\{1\})$.\\

The \emph{geodesic flow} $\phi^t: ST^2 \to ST^2$ is the restriction of the Euler-Lagrange flow of the 
Lagrangian $L_F$, with 
$$
L_F = \frac{1}{2}F^2
$$
to the unit tangent bundle. A \emph{geodesic} we call either a trajectory of the Euler-Lagrange flow, or 
its projection to $T^2$, i.e. a geodesic is a curve $t \mapsto c(t) \subset T^2$ satisfying the Euler-Lagrange equation
$$
\partial_x L_F(c, \dot c) - \partial_t(\partial_v L_F(c, \dot c)) = 0
$$

To embed the map $\hat f$ into the geodesic flow we use a theorem of Moser \cite{moser86} to express 
$\hat f$ as the time-1 map of a strictly convex, time-periodic Hamiltonian on $S^1$. 

\begin{thm} \emph{(J. Moser '86)} \label{moserthm} Given a $C^\infty$ twist map $f: Z \to Z$ with $f(x,y) = (x + c \cdot y, y)$ for
large $|y|$, there exists a strictly convex, time-periodic Hamiltonian $H$ on $S^1$, such that the time-1 map 
$\psi^{0,1}_H: Z \to Z$ agrees with $f$. 
\end{thm}

\begin{rmk}
Moser's original theorem is formulated for twist maps $f$ on the closed annulus $A = S^1 \times [0,1]$. In
the proof in \cite{moser86} the map $f$ is extended to a twist map on the cylinder $Z$ with $f(x,y) = (x + c \cdot y, y)$ for
$|y| > D$ for a positive constant $D$. There exist constants $D_+, D_- \in \R$, such that the Hamiltonian $H$ is equal to 
$\tfrac{1}{2}y^2 + D_\pm$ for large values of $|y|$, depending on whether $y > D$  or $y < -D$. 
\end{rmk}

Let $H$ be the Hamiltonian obtained from theorem \ref{moserthm} generating the previously constructed twist map $\hat f$. We lift $H$ to 
obtain a $e_1 \cdot \Z$-periodic Hamiltonian $H$ on $\R^2$. The Legendre transformation 
$\Leg_t: \R^2 \to \R^2$ is a global diffeomorphism and agrees with the identity for large values of $|y|$. 
Thus, we obtain an associated Lagrangian $\hat L: S^1 \times \R^2 \to \R$ with 
$$
\hat L(t,x,y) = \frac{1}{2} y^2 - D_\pm \quad \text{for } |y| > D
$$

Observe that the time-dependent Euler-Lagrange flow of $\hat L$ is complete (i.e. the 
Euler-Lagrange solutions exist for all times) because the sets $\{ y = \const \}$ are invariant 
for large $|y|$.
To embed $\hat L$ into a Finsler metric we need to perturb it for large values of $|y|$. Let $h_+, h_-: \R \to \R$
be smooth functions with 
$$
h_+(y) = 
\left\{
	\begin{array}{ll}
		\tfrac{1}{2}y^2 - D_+  & \mbox{if } y < D+1 \\
		\sqrt{A + By^2} & \mbox{if } y > D+2
	\end{array}
\right.
$$
and
$$
h_-(y) = 
\left\{
	\begin{array}{ll}
		\sqrt{A + By^2}  & \mbox{if } y < -D -2 \\
		\tfrac{1}{2}y^2 - D_- & \mbox{if } y > -D-1
	\end{array}
\right.
$$
With constants $A,B > 0$ chosen in such a way that it is possible to choose $h_+, h_-$ with 
$h_\pm'' > 0$. 
We define a Lagrangian $L$ via 
$$
L(t,x,y) = 
\left\{
	\begin{array}{ll}
		h_-(y) & \mbox{if } y < -D \\
		\hat L(t,x,y) & \mbox{if } -D \leq y \leq D\\
		h_+(y) & \mbox{if } y > D
	\end{array}
\right.
$$
Observe that the time-dependent flow of $\hat L$ is the same as the time-dependent flow of $L$ since 
$\hat L$ and $L$ only differ where they are both only dependent on $y$. 
Let $F_0$ be a Finsler metric on $T^2$ given by 
$$
F_0(t,x,v_1,v_2) = \sqrt{Av_1^2 + Bv_2^2}
$$
The Lagrangian $L$ is chosen in such a way that $L(t,x,y) = F_0(t,x,1,y)$ for large values of $|y|$.
We define a map $F$ on $T^2$ via
$$
F(t,x,v_1,v_2) = 
\left\{
	\begin{array}{ll}
		v_1 \cdot L(t,x,\tfrac{v_2}{v_1})  & \mbox{if } v_1 > 0 \\
		F_0(t,x,v_1,v_2) & \mbox{if } v_1 \leq 0
	\end{array}
\right.
$$

\begin{rmk}

From the proof of Moser's theorem \ref{moserthm} it follows that if $\hat f$ is chosen $C^\infty$-close to 
the shear map then the obtained Hamiltonian will be $C^\infty$-close to $H = \tfrac{1}{2} y^2$.  Consequently, 
the above obtained Lagrangian $L$ can be chosen to be $C^\infty$-close to a function $h$ only dependent
on $y$ with 
$$
h(y) = 
\left\{
	\begin{array}{ll}
		\tfrac{1}{2}y^2 & \mbox{if } |y| < D + 1 \\
		\sqrt{A + By^2} & \mbox{if } |y| > D +2
		
	\end{array}
\right.
$$
Thus, for any compact subset $K$ of $TT^2$  we can find a sequence $\hat f_i$ of twist maps converging
to $f_1$ and do the above construction, such that the resulting Finsler metrics $F_i$ become 
arbitrarily $C^\infty$-close to the flat metric 

$$
\bar F(t,x,v_1,v_2) = 
\left\{
	\begin{array}{ll}
		v_1 \cdot h(\tfrac{v_2}{v_1})  & \mbox{if } v_1 > 0 \\
		F_0(t,x,v_1,v_2) & \mbox{if } v_1 \leq 0
	\end{array}
\right.
$$
on the set $K$. 

\end{rmk}

The following two propositions are due to J.P. Schr\"oder \cite{schroeder13}. We include their proofs for 
completeness.

\begin{prop}
$F$ defines a $C^\infty$ Finsler metric on $T^2$. 
\end{prop}

\begin{proof}
Regularity and Positive homogeneity follow directly from the definition of $F$.

We check strict convexity in each fiber. Let $(t,x) \in T^2$ be fixed and define 
$f: \R_{>0} \times \R \rightarrow \R$ via 
$$
f(v_1, v_2) := v_1 \cdot l(v_2/v_1)
$$
where $l(y) := L(t,x,y)$ for every $y \in \R$. 
We compute the derivatives
$$
\partial_1f(u_1,u_2) = l\left( \frac{u_2}{u_1}\right) - \frac{u_2}{u_1} \cdot l' \left( \frac{u_2}{u_1}\right)
$$
$$
\partial_2f(u_1,u_2) =  l' \left( \frac{u_2}{u_1}\right)
$$
second derivatives 
$$
\partial_{11}f(u_1,u_2) = \frac{u_2^2}{u_1^3} \cdot  l'' \left( \frac{u_2}{u_1}\right)
$$
$$
\partial_{12}f(u_1,u_2) = \partial_{21}f(u_1,u_2) = - \frac{u_2}{u_1^2}\cdot  l'' \left( \frac{u_2}{u_1}\right)
$$
$$
\partial_{22}f(u_1,u_2) = \frac{1}{u_1} \cdot l'' \left( \frac{u_2}{u_1}\right)
$$

Consequently, we have for $u = (u_1,u_2) \in \R_{>0} \times \R$ and $v = (v_1,v_2) \in \R^2$
$$
\langle v , Hess f(u)v \rangle = \left(v_1 \cdot \frac{u_2}{u_1}  - v_2\right)^2 \cdot \frac{l''(u_2/u_1)}{u_1}
$$

To see that $F$ is strictly convex we have to check that fiberwise the Hessian of $L_F = \frac{1}{2} F^2$ is 
positive definite. For $(t,x) \in T^2$ fixed let $L_F: T_{(t,x)}T^2 \to \R$ be given by $L_F(u) = \tfrac{1}{2}F(t,x,u)^2$.
For $u \in \R_{>0} \times \R$ we then have 
$$
L_F(u) = \frac{1}{2}(u_1 \cdot l(u_2/u_1))^2
$$
We compute partial derivatives
$$
\partial_1 L_F(u) = f(u) \cdot \partial_1f(u)
$$
$$
\partial_2 L_F(u) = f(u) \cdot \partial_2f(u)
$$
$$
\partial_{11} L_F(u) = (\partial_1f(u))^2 + f(u) \cdot \partial_{11}f(u)
$$
$$
\partial_{12} L_F(u) = \partial_1f(u)\cdot \partial_2f(u) + f(u) \cdot \partial_{12}f(u)
$$
$$
\partial_{22} L_F(u) = (\partial_2f(u))^2 + f(u) \cdot \partial_{22}f(u)
$$

from this we get that
$$
Hess L_F(u) = A + f(u) \cdot Hess f(u)
$$
where
$$
A = 
\begin{pmatrix}
(\partial_1f(u))^2 & \partial_1f(u)\cdot \partial_2f(u) \\
\partial_1f(u)\cdot \partial_2f(u) & (\partial_2f(u))^2
\end{pmatrix}
$$

For $v \in \R^2$ we compute 
\begin{flalign*}
v^T A v &= v_1^2(\partial_1f(u))^2 + 2v_1v_2\partial_1f(u)\cdot \partial_2f(u) + v_2^2 (\partial_2f(u))^2\\
&= (v_1 \partial_1f(u) + v_2 \partial_2f(u))^2 \\
&= (Df(u)v)^2
\end{flalign*}

Hence we have

\begin{flalign*}
v^T HessL_F(u) v &= v^T (A + f(u) Hess f(u)) v \\
&= v^T A v +  f(u) v^T Hess f(u) v \\
&=\underbrace{(Df(u)v)^2}_a + \underbrace{f(u) \left(v_1 \frac{u_2}{u_1} - v_2\right)^2 \frac{l''(u_2/u_1)}{u_1}}_b
\end{flalign*}

Observe that $a$ and $b$ are each $\geq 0$. Assume now, that $b = 0$. Since $f, l''$ and $u_1$ are $>0$ it follows
that $v_1 \frac{u_2}{u_1} - v_2 = 0$. From this it follows that $v = \lambda \cdot u$ are linearly dependent. In this
case we have $Df(u)v = \lambda \cdot Df(u)u = f(v) >0$. 
Hence the Hessian $Hess L_F(x,u)$ is positive definite for $u \in \R_{>0} \times \R$. It is also positive definite for
$u \in \R_{\leq 0} \times \R - \{0\}$ since $L$ coincides there with the Finsler metric $F_0$.

\end{proof}

\begin{prop} \label{schroeder2}
Let $\theta: \R \to \R$ be a smooth function and let $\gamma: \R \to \R^2$ be the curve given by
$$
\gamma(t) = (t, \theta(t))
$$
Then $\gamma$ is a reparametrization of a lifted $F$-geodesic if and only if $\theta$ is an Euler-Lagrange
solution of $L$ (seen as a Lagrangian lifted to $\R$).
\end{prop}

\begin{proof}
Observe that we have the following relation between the Lagrangian action $A_{L}$ and the Finsler length $l_F$.
\begin{flalign*}
A_{L}(\theta|_{[a,b]}) &= \int_a^b L(t, \theta(t), \theta'(t)) dt \\
&= \int_a^b F(t, \theta(t), 1, \theta'(t)) dt \\
&= \int_a^b F(\gamma(t), \dot \gamma(t)) dt \\
&= l_F(\gamma|_{[a,b]})
\end{flalign*}

Assume now that $\gamma:[a,b] \to \R^2$ is a reparametrization of an $F$-geodesic, 
i.e. $\partial_{s=0} l_F(\gamma_s) =0$ for any
proper variation of $\gamma$. Let $\theta_s : [a,b] \to \R$  be a proper variation of $\theta$. 
From $\theta_s$ we 
construct a proper variation $\gamma_s$ of $\gamma$ via 
$$
\gamma_s(t) = (t,\theta_s(t))
$$ 
Then we have $A_{L}(\theta_s) = l_F(\gamma_s)$ for every $s$, and hence we have 
$$
\partial_s|_{s=0} A_{L}(\theta_s) = \partial_s|_{s=0}l_F(\gamma_s) = 0.
$$
This proves one direction.
To prove the other direction assume now that $\theta: [a,b] \to \R$ is critical with respect to the Lagrangian action and 
let $X: [a,b] \to \R^2$ be a vector field along $\gamma$ with $X(a) = X(b) = 0$. Since 
$\dot \gamma(t) = (1,\theta'(t))$ the pair of vectors $\{\dot \gamma(t) , e_2 \}$ always forms a basis of $\R^2$. 
Thus, we can rewrite the vector field $X$ as 
$$
X(t) = \underbrace{\lambda(t) \dot \gamma(t)}_{A(t)} + \underbrace{\mu(t) e_2}_{B(t)}
$$
for functions $\lambda, \mu$ with $\lambda(a) = \lambda(b) = \mu(a) = \mu(b) = 0$.
Let $\gamma_s$ be a proper variation of $\gamma$ corresponding to the variational vector field $X$ and let $\beta_s$
be a proper variation of $\gamma$ corresponding to $B$, i.e.
$$
\partial_s|_{s=0} \gamma_s(t) = X(t) \quad \text{and} \quad \partial_s|_{s=0} \beta_s(t) = B(t)
$$
Observe that for small $|s|$ the curve $\eta_s: t \mapsto \gamma(t + s\lambda(t))$ is a reparametrization of $\gamma$ 
and hence has length independent of $s$. Thus

\begin{flalign*}
0 &= \partial_s|_{s=0} l_F(\eta_s) \\
&=  \partial_s|_{s=0} \int_a^b F(\eta_s(t), \partial_t \eta_s(t)) dt \\ 
&=  \int_a^b \partial_s|_{s=0} F(\eta_s(t), \partial_t \eta_s(t)) dt \\ 
&= \int_a^b  \partial_1 F(\eta_0(t), \partial_t \eta_0(t)) \partial_s|_{s=0} \eta_s(t) 
     + \partial_2 F(\eta_0(t), \partial_t \eta_0(t)) \partial_s|_{s=0} \partial_t \eta_s(t) dt \\ 
&= \int_a^b  \partial_1 F(\gamma(t), \partial_t \gamma(t)) \partial_s|_{s=0} \eta_s(t) 
     + \partial_2 F(\gamma(t), \partial_t \gamma(t)) \partial_t  \partial_s|_{s=0} \eta_s(t) dt \\ 
&= \int_a^b  \partial_1 F(\gamma(t), \partial_t \gamma(t)) A(t)
     + \partial_2 F(\gamma(t), \partial_t \gamma(t)) \dot A(t) dt 
\end{flalign*}

Hence it follows that 

\begin{flalign*}
\partial_s|_{s=0} l_F(\gamma_s) &= \partial_s|_{s=0} \int_a^b F(\gamma_s(t), \dot \gamma_s(t)) dt \\
&= \int_a^b \partial_1 F(\gamma_0(t), \dot \gamma_0(t)) \partial_s|_{s=0} \gamma_s(t) 
+ \partial_2 F(\gamma_0(t), \dot \gamma_0(t)) \partial_s|_{s=0} \partial_t \gamma_s(t) \\
&= \int_a^b \partial_1 F(\gamma(t), \dot \gamma(t)) \partial_s|_{s=0} \gamma_s(t) 
+ \partial_2 F(\gamma(t), \dot \gamma(t)) \partial_t \partial_s|_{s=0} \gamma_s(t) \\
&= \int_a^b \partial_1 F(\gamma(t), \dot \gamma(t)) X(t) 
+ \partial_2 F(\gamma(t), \dot \gamma(t)) \dot X(t) \\
&= \int_a^b \partial_1 F(\gamma(t), \dot \gamma(t)) (A(t) + B(t))
+ \partial_2 F(\gamma(t), \dot \gamma(t)) (\dot A(t) + \dot B(t)) \\
&= \int_a^b \partial_1 F(\gamma(t), \dot \gamma(t)) B(t)
+ \partial_2 F(\gamma(t), \dot \gamma(t)) \dot B(t) \\
&= \partial_s|_{s=0} l_F(\beta_s)
\end{flalign*}

Consequently, the curve $\gamma$ is critical with respect to the Finsler length if
$$
\partial_s |_{s=0} l_F(\beta_s) = 0
$$
for every proper variation $\beta_s$, which varies $\gamma$ only in $e_2$ direction, i.e. $\beta_s$ is of the form
$$
\beta_s(t) = (t,\theta_s(t))
$$
For those variations we have already computed that if $\theta_s$ is critical with respect to the $L$-action then 
$\beta_s$ is critical with respect to the Finsler length.

\end{proof}

We define sets $V, V_0 \subset ST^2$ via 
$$
V = \{ (x,v) \in ST^2 \mid x \in T^2, v_1 > 0 \}, \quad V_0 = \{ (0,h, v) \in ST^2 \mid h \in S^1, v_1 > 0 \}
$$ 

From the completeness of the time-dependent Euler-Lagrange flow of $L$ and proposition 
\ref{schroeder2} it follows that the lifts to $\R^2$ of geodesics $c_v$ with $v_1 > 0$ 
are graphs over the euclidean line $\R \cdot e_1 \subset \R$ and pass through the section $V_0$ after
finite time. 
Thus, the first return map $R: V_0 \to V_0$ of the geodesic flow is well-defined.

\begin{prop}
The return map $R: V_0 \to V_0$ is conjugated via a diffeomorphism to the twist map $\hat f$.
\end{prop}

\begin{proof}
To see that the return map $R$ is conjugated to $\hat f$  observe that $R$ is given by
\begin{equation} \label{eq:returnmap}
R(0,h,v_1,v_2) = \left( 0, \theta(1), \frac{(1,\theta'(1))}{F(0,\theta(1),1,\theta'(1))}\right)
\end{equation}
where $\theta: \R \to \R$ is the Euler-Lagrange solution of the lifted Lagrangian $L$ with 
$\theta(0) = h$ and $\theta'(0) = \tfrac{v_2}{v_1}$. This is true because after proposition 
\ref{schroeder2} the curve $\gamma: \R \to \R^2$ with 
$\gamma(t) = (t, \theta(t))$ is a reparametrized lift of the geodesic $c: \R \to T^2$ with initial values
$c(0) = (0,h)$ and $\dot c(0) = (v_1,v_2)$. The reparametrized lift $\gamma$ passes through a translate
$V_0 + e_1$ of $V_0$ again for the first time at time $t=1$ and thus the return map $R$ maps 
$(0,h,v_1,v_2)$ to $(\gamma(1), \tfrac{\dot \gamma(1)}{F(\gamma(1),\dot \gamma(1)})$, which is equal 
to the expression in equation \eqref{eq:returnmap}.
The diffeomorphism $g: V_0 \to Z$ conjugating $R$ and $\hat f$ is given by
$$
g(0,h,v_1,v_2) = \left( h, \frac{v_2}{v_1}\right)
$$
with inverse
$$
g^{-1}(x,y) = \left( 0, x, \frac{(1,y)}{F(0,x,1,y)}\right)
$$

\end{proof}

\begin{prop}
The first return map $R$ has positive metric entropy.
\end{prop}

\begin{proof}
Let $g: V_0 \to Z$ be the diffeomorphism conjugating $R$ to $\hat f$. Let $I \subset Z$ denote 
the stochastic island for $\hat f$, i.e. every point $x \in I$ has positive maximal Lyapunov exponent
and $\Area(I) = \int_I | dx \wedge dy | > 0$. From the conjugacy of $\hat f$ and $R$ it follows that the 
maximal Lyapunov exponent of every $v \in g^{-1}(I)$ remains positive. Let $\omega$ denote the area form of $V_0$ obtained by
restricting the differential $d\lambda$ of the standard Liouville form $\lambda$ to $V_0$. The return
map $R$ preserves $\omega$. Since the pullback $g^*(dx \wedge dy)$ and $\omega$ are both 
area forms there is a positive function $j: V_0 \to \R_{>0}$, such that
$$
j \cdot \omega = g^*(dx \wedge dy)
$$
Since $j$ is positive we have that $\int_{g^{-1}(I)} \omega \neq 0$ if and only if 
$\int_{g^{-1}(I)} j \cdot  \omega \neq 0$. And since
$$
\int_{g^{-1}(I)} j \cdot \omega = \int_{g^{-1}(I)} g^*(dx \wedge dy) = \int_I dx \wedge dy
$$
we have that $\Area(g^{-1}(I)) > 0$. Consequently, $R$ has positive metric entropy. 
\end{proof}



\begin{thebibliography}{9}

\bibitem{bergerturaev19}
  P. Berger, D. Turaev,
  \textit{On Herman's Positive Entropy Conjecture},
  Advances in Mathematics 349, 1234 - 1288 (2019)

\bibitem{herman83}
  M. Herman,
  \textit{Sur les courbes invariantes par les diff\'eomorphismes de l'anneau},
  vol. I, Aste\'erisque, 103-104 (1983)

\bibitem{moser86}
  J. Moser,
  \textit{Monotone twist mappings and the calculus of variations},
  Ergodic Theory and Dynamical Systems, Vol. 6 (1986)

\bibitem{schroeder13}
  J. P. Schr\"oder, 
  \textit{Tonelli Lagrangians on the 2-torus: global minimizers, invariant tori and topological entropy},
  Ph.D. thesis, Ruhr-Universit\"at Bochum (2013)
 

\end{thebibliography}
\end{document}